\documentclass[12pt,reqno]{article}
\usepackage{caption}
\usepackage{diagbox}
\usepackage{verbatim}
\usepackage[usenames]{color}
\usepackage{amssymb}
\usepackage{graphicx}
\usepackage{amscd}

\usepackage{enumerate}

\usepackage{tikz}
\usetikzlibrary{decorations.text}

\usepackage[colorlinks=true,
linkcolor=webgreen,
filecolor=webbrown,
citecolor=webgreen]{hyperref}

\definecolor{webgreen}{rgb}{0,.5,0}
\definecolor{webbrown}{rgb}{.6,0,0}

\usepackage{color}
\usepackage{fullpage}
\usepackage{float}

\usepackage{graphics,amsmath,amssymb}
\usepackage{amsthm}
\usepackage{amsfonts}
\usepackage{latexsym}
\usepackage{epsf}

\DeclareMathOperator{\IB}{LBP}

\usepackage{breakurl}

\begin{document}

\theoremstyle{plain}
\newtheorem{theorem}{Theorem}
\newtheorem{corollary}[theorem]{Corollary}
\newtheorem{lemma}[theorem]{Lemma}
\newtheorem{proposition}[theorem]{Proposition}

\theoremstyle{definition}
\newtheorem{definition}[theorem]{Definition}
\newtheorem{example}[theorem]{Example}
\newtheorem{conjecture}[theorem]{Conjecture}

\theoremstyle{remark}
\newtheorem{remark}[theorem]{Remark}

\author{Daniel Gabric\footnote{
Department of Math/Stats,
University of Winnipeg,
Winnipeg, MB R3B 2E9,
Canada; {\tt d.gabric@uwinnipeg.ca}.}
\quad and Jeffrey Shallit\footnote{School of Computer Science, University of Waterloo, Waterloo, ON  N2L 3G1, Canada; {\tt shallit@uwaterloo.ca}.}}

\date{}

\title{Smallest and Largest Block Palindrome Factorizations}

\maketitle

\begin{abstract}
A \emph{palindrome} is a word that reads the same forwards and backwards. A \emph{block palindrome factorization} (or \emph{BP-factorization}) is a factorization of a word into blocks that becomes palindrome if each identical block is replaced by a distinct symbol. We call the number of blocks in a BP-factorization the \emph{width} of the BP-factorization. The \emph{largest BP-factorization} of a word $w$ is the BP-factorization of $w$ with the maximum width. We study words with certain BP-factorizations. First, we give a recurrence for the number of length-$n$ words with largest BP-factorization of width $t$. Second, we show that the expected width of the largest BP-factorization of a word tends to a constant. Third, we give some results on another extremal variation of BP-factorization, the \emph{smallest BP-factorization}. A \emph{border} of a word $w$ is a non-empty word that is both a proper prefix and suffix of $w$. Finally, we conclude by showing a connection between words with a unique border and words whose smallest and largest BP-factorizations coincide.
\end{abstract}

\section{Introduction}
Let $\Sigma_k$ denote the alphabet $\{0,1,\ldots, k-1\}$. The length of a word $w$ is denoted by $|w|$. A \emph{border} of a word $w$ is a non-empty word that is both a proper prefix and suffix of $w$. A word is said to be \emph{bordered} if it has a border. Otherwise, the word is said to be \emph{unbordered}. For example, the French word {\tt entente} is bordered, and has two borders, namely {\tt ente} and {\tt e}.

It is well-known~\cite{Nielsen:1973} that the number $u_n$ of length-$n$ unbordered words over $\Sigma_k$ satisfies
\begin{equation}
u_n =
\begin{cases} 
      1, & \text{if }n=0;\\
      ku_{n-1} - u_{n/2}, & \text{if $n>0$ is even;}\\
      ku_{n-1}, & \text{if $n$ is odd.}
   \end{cases}\label{equation:unbordered}
\end{equation}

A \emph{palindrome} is a word that reads the same forwards as it does backwards. More formally, letting $w^R = w_nw_{n-1}\cdots w_1$ where $w=w_1w_2\cdots w_n$ and all $w_i$ are symbols, a palindrome is a word $w$ such that $w= w^R$. The definition of a palindrome is quite restrictive. The second half of a palindrome is fully determined by the first half. Thus, compared to all length-$n$ words, the number of length-$n$ palindromes is vanishingly small. But many words exhibit palindrome-like structure. For example, take the English word ${\tt marjoram}$. It is clearly not a palindrome, but it comes close. Replacing the block ${\tt jo}$ with a single letter turns the word into a palindrome. In this paper, we consider a generalization of palindromes that incorporates this kind of palindromic structure.

In the 2015 British Olympiad~\cite{Olympiad:2015}, the concept of a block palindrome factorization was first introduced. Let $w$ be a non-empty word. A \emph{block palindrome factorization} (or \emph{BP-factorization}) of $w$ is a factorization $w=w_{m}\cdots w_{1}w_0w_1\cdots w_m$ of a word such that $w_0$ is a possibly empty word, and every other factor $w_i$ is non-empty for all $i$ with $1\leq i \leq m$. We say that a BP-factorization $w_{m}\cdots w_{1}w_0w_1\cdots w_m$ is of \emph{width} $t$ where $t=2m+1$ if $w_0$ is non-empty and $t=2m$ otherwise. In other words, the width of a BP-factorization is the number of non-empty blocks in the factorization. The \emph{largest BP-factorization}\footnote{Largest BP-factorizations also appear in \url{https://www.reddit.com/r/math/comments/ga2iyo/i_just_defined_the_palindromity_function_on/}.}~\cite{Goto&etal:2018} of a word $w$ is a BP-factorization $w=w_{m}\cdots w_{1}w_0w_1\cdots w_m$ where $m$ is maximized (i.e., where the width of the BP-factorization is maximized). See~\cite{Mahalingam&Maity&Pandoh&Raghavan:2021, Mahalingam&Maity&Pandoh:2023} for more on the topic of BP-factorizations and block reversals. Kolpakov and Kucherov~\cite{Kolpakov&Kucherov:2009} studied a special case of BP-factorizations, the \emph{gapped palindrome}. If $w_0$ is non-empty and $|w_{i}|=1$ for all $i$ with $1\leq i \leq m$, then $w$ is said to be a \emph{gapped palindrome}. R\'egnier~\cite{Regnier:1992} studied something similar to BP-factorizations, but in her paper she was concerned with borders of borders. See~\cite{Frid&Puzynina&Zamboni:2013, Ravsky:2003} for results on factoring words into palindromes.

\begin{example}\label{example:abra}
We use the centre dot $\cdot$ to denote the separation between blocks in the BP-factorization of a word.

Consider the word {\tt abracadabra}. It has the following BP-factorizations:
\begin{gather*}
    {\tt abracadabra}, \\
    {\tt abra}\cdot {\tt cad} \cdot {\tt abra}, \\
    {\tt a}\cdot {\tt bracadabr}\cdot {\tt a}, \\
    {\tt a}\cdot {\tt br}\cdot {\tt acada}\cdot {\tt br}\cdot {\tt a}, \\
    {\tt a} \cdot {\tt br} \cdot {\tt a} \cdot {\tt cad} \cdot {\tt a} \cdot {\tt br} \cdot {\tt a}.
\end{gather*}
The last BP-factorization is of width $7$ and has the longest width; thus it is the largest BP-factorization of {\tt abracadabra}.
\end{example}

Let $w$ be a length-$n$ word. Suppose $w_{m}\cdots w_{1}w_0w_1\cdots w_m$ is the largest BP-factorization of $w$. Goto et al.~\cite{Goto&etal:2018} showed that $w_{i}$ is the shortest border of $w_{i}\cdots w_{1}w_0w_1\cdots w_i$ where $i\geq 1$. This means that we can compute the largest BP-factorization of $w$ by greedily ``peeling off" the shortest borders of central factors until you hit an unbordered word or the empty word.

The rest of the paper is structured as follows. In Section~\ref{section:countIB} we give a recurrence for the number of length-$n$ words with largest BP-factorization of width $t$. In Section~\ref{section:expIB} we show that the expected width of the largest BP-factorization of a length-$n$ word tends to a constant. In Section~\ref{section:smallIB} we consider \emph{smallest BP-factorizations} in the sense that one ``peels off" the longest non-overlapping border. We say a border $u$ of a word $w$ is \emph{non-overlapping} if $|u| \leq |w|/2$; otherwise $u$ is \emph{overlapping}. Finally, in Section~\ref{section:uniqueIB} we present some results on words with a unique border and show that they are connected to words whose smallest and largest BP-factorizations are the same.

\section{Counting largest BP-factorizations}\label{section:countIB}
In this section, we prove a recurrence for the number $\IB_k(n,t)$ of length-$n$ words over $\Sigma_k$ with largest BP-factorization of width $t$. See Table~\ref{table:IBTable} for sample values of $\IB_2(n,t)$ for small $n$, $t$. For the following theorem, recall the definition of $u_n$ from Equation~\ref{equation:unbordered}.

\begin{theorem}Let $n,t \geq 0$, and $k\geq 2$ be integers. Then
\[
\IB_k(n,t)=    \begin{cases} 
      \sum_{i=1}^{(n-t)/2+1} u_{i} \IB_k(n-2i, t-2), & \text{if $n$, $t$ even;} \\
      \sum_{i=1}^{(n-t+1)/2} u_{2i} \IB_k(n-2i,t-1), & \text{if $n$ even, $t$ odd;}\\
      0, & \text{if $n$ odd, $t$ even;} \\
      \sum_{i=1}^{(n-t)/2+1}u_{2i-1}\IB_k(n-2i+1,t-1), & \text{if $n$, $t$ odd.}\\
   \end{cases}
\]

where 
\begin{align}
\IB_k(0,0) &= 1,\nonumber \\
\IB_k(2n,2) &= u_n, \nonumber \\
\IB_k(n,1) &= u_n. \nonumber 
\end{align}
\end{theorem}
\begin{proof}\sloppy
Let $w$ be a length-$n$ word whose largest BP-factorization $w_{m}\cdots w_{1}w_0w_1\cdots w_{m}$ is of width $t$. Clearly $\IB_k(0,0)=1$. We know that each block in a largest BP-factorization is unbordered, since each block is a shortest border of some central factor. This immediately implies $\IB_k(n,1)=u_n$ and $\IB_k(2n, 2)= u_n$.

Now we take care of the other cases.

\begin{itemize}
    \item Suppose $n$, $t$ are even. Then by removing both instances of $w_1$ from $w$, we get $w' = w_{m}\cdots w_{2}w_2\cdots w_m$, which is a length-$(n-2|w_1|)$ word whose largest BP-factorization is of width $t-2$. This mapping is clearly reversible, since all blocks in a largest BP-factorization are unbordered, including $w_1$. Thus summing over all possible $w_1$ and all length-$(n-2|w_1|)$ words with largest BP-factorization of width $t-2$ we have \[\IB_k(n,t) = \sum_{i=1}^{(n-t)/2+1} u_{i} \IB_k(n-2i, t-2).\]
    \item Suppose $n$ is even and $t$ is odd. Then by removing $w_0$ from $w$, we get $w' = w_{m}\cdots w_{1}w_1\cdots w_m$, which is a length-$(n-|w_0|)$ word whose largest BP-factorization is of width $t-1$. This mapping is reversible for the same reason as in the previous case. The word $w'$ is of even length since $|w'| = 2|w_1\cdots w_m|$.  Since $n$ is even and $|w'|$ is even, we must have that $|w_0|$ is even as well. Thus summing over all possible $w_0$ and all length-$(n-|w_0|)$ words with largest BP-factorization of width $t-1$, we have
    \[\IB_k(n,t) = \sum_{i=1}^{(n-t+1)/2} u_{2i} \IB_k(n-2i,t-1).\]
    \item Suppose $n$ is odd and $t$ is even.
    Then the length of $w$ is $2|w_1\cdots w_m|$, which is even, a contradiction. Thus $\IB_k(n,t)=0$.
    \item Suppose $n$, $t$ are odd. Then by removing $w_0$ from $w$, we get $w' = w_{m}\cdots w_{1}w_1\cdots w_m$, which is a length-$(n-|w_0|)$ word whose largest BP-factorization is of width $t-1$. This mapping is reversible for the same reasons as in the previous cases. Since $n$ is odd and $|w'|$ is even (proved in the previous case), we must have that $|w_0|$ is odd. Thus summing over all possible $w_0$ and all length-$(n-|w_0|)$ words with largest BP-factorization of width $t-1$, we have
    \[\sum_{i=1}^{(n-t)/2+1}u_{2i-1}\IB_k(n-2i+1,t-1).\]
\end{itemize}
\end{proof}

\begin{table}[H]

\centering
\begin{tabular}{|c|cccccccccc|}
\hline
\backslashbox{$n$}{$t$}  & $1$ & $2$ & $3$ & $4$ & $5$ & $6$ & $7$ & $8$ & $9$ & $10$  \\
\hline
10 & 284 & 12 & 224 & 40 & 168 & 72 & 96 & 64 & 32 & 32\\ 
11 & 568 & 0 & 472 & 0 & 416 & 0 & 336 & 0 & 192 & 0\\ 
12 & 1116 & 20 & 856 & 88 & 656 & 176 & 448 & 224 & 224 & 160\\ 
13 & 2232 & 0 & 1752 & 0 & 1488 & 0 & 1248 & 0 & 896 & 0\\ 
14 & 4424 & 40 & 3328 & 176 & 2544 & 432 & 1856 & 640 & 1152 & 640\\ 
15 & 8848 & 0 & 6736 & 0 & 5440 & 0 & 4576 & 0 & 3584 & 0\\ 
16 & 17622 & 74 & 13100 & 372 & 9896 & 984 & 7408 & 1744 & 5088 & 2080\\ 
17 & 35244 & 0 & 26348 & 0 & 20536 & 0 & 16784 & 0 & 13664 & 0\\ 
18 & 70340 & 148 & 51936 & 760 & 38824 & 2248 & 29152 & 4416 & 21088 & 6240\\ 
19 & 140680 & 0 & 104168 & 0 & 79168 & 0 & 62800 & 0 & 51008 & 0\\ 
20 & 281076 & 284 & 206744 & 1592 & 153344 & 4992 & 114688 & 10912 & 84704 & 17312\\ 
\hline
\end{tabular}
\captionsetup{justification=centering}
\caption{Some values of $\IB_2(n,t)$ for $n$, $t$ where $10 \leq n \leq 20$ and $1\leq t \leq 10$.}
\label{table:IBTable}
\end{table}

\section{Expected width of largest BP-factorization}\label{section:expIB}
In this section, we show that the expected width $E_{n,k}$ of the largest BP-factorization of a length-$n$ word over $\Sigma_k$ is bounded by a constant. From the definition of expected value, it follows that \[ E_{n,k}=\frac{1}{k^n}\sum_{i=1}^n i\cdot \IB_k(n,i).\]Table~\ref{table:IBExpTable} shows the behaviour of $\lim\limits_{n\to \infty}E_{n,k}$ as $k$ increases.

\begin{lemma}\label{lemma:IBconv}
Let $k\geq 2$ and $n\geq t \geq 1$ be integers. Then \[ \frac{\IB_k(n,t)}{k^n}\leq\frac{1}{k^{t/2-1}}.\]
\end{lemma}
\begin{proof}Let $w$ be a length-$n$ word whose largest BP-factorization $w_{m}\cdots w_{1}w_0w_1\cdots w_{m}$ is of width $t$. Since $w_i$ is non-empty for every $1\leq i \leq m$, we have that $\IB_k(n,t) \leq k^{n-m} \leq k^{n-t/2+1}$. So \[\frac{\IB_k(n,t)}{k^n}\leq \frac{1}{k^{t/2-1}}\] for all $n\geq t\geq 1$.
\end{proof}
\begin{theorem}
The limit $E_k=\lim\limits_{n\to \infty} E_{n,k}$ exists for all $k\geq 2$.
\end{theorem}
\begin{proof}
Follows from the definition of $E_{n,k}$, Lemma~\ref{lemma:IBconv}, and the direct comparison test for convergence.
\end{proof}
\noindent 
Interpreting $E_k$ as a power series in $k^{-1}$, we empirically observe that $E_k$ is approximately equal to
\[1 + \frac{2}{k} + \frac{4}{k^2}+\frac{6}{k^3} + \frac{10}{k^4}+\frac{16}{k^5} + \frac{24}{k^6} + \frac{38}{k^7}+\frac{58}{k^8}+\frac{88}{k^9}+\cdots.\]
We conjecture the following about $E_k$.
\begin{conjecture}
Let $k\geq 2$. Then
\[E_k = 1+ \sum_{i=1}^\infty a_i k^{-i}\]
where the sequence $(a_n/2)_{n\geq 1}$ is \href{https://oeis.org/A274199}{\underline{A274199}} in the \emph{On-Line Encyclopedia of Integer Sequences} (OEIS)~\cite{OEIS}.
\end{conjecture}

\begin{table}[H]
\centering
\begin{tabular}{|c|c|}
\hline
$k$ & $\approx E_k$ \\
\hline
2 & 6.4686 \\ 
3 & 2.5908 \\ 
4 & 1.9080  \\ 
5 & 1.6314 \\ 
6 & 1.4827 \\ 
7 & 1.3902 \\ 
8 & 1.3272 \\ 
9 & 1.2817 \\ 
10 &1.2472 \\ 
\vdots & \vdots \\
100 & 1.0204\\
\hline
\end{tabular}
\captionsetup{justification=centering}
\caption{Asymptotic expected width of a word's largest BP-factorization.}
\label{table:IBExpTable}
\end{table}
Cording et al.~\cite{Cording&Gagie&Knudsen&Kociumaka:2021} proved that the expected length of the longest unbordered factor in a word is $\Theta(n)$. Taking this into account, it is not surprising that the expected length of the largest BP-factorization of a word tends to a constant.

\section{Smallest BP-factorization}\label{section:smallIB}
A word $w$, seen as a block, clearly satisfies the definition of a BP-factorization. Thus, taken literally, the smallest BP-factorization for all words is of width $1$. But this is not very interesting, so we consider a different definition instead.  A border $u$ of a word $w$ is \emph{non-overlapping} if $|u| \leq |w|/2$; otherwise $u$ is \emph{overlapping}. We say that the \emph{smallest BP-factorization} of a word $w$ is a BP-factorization $w=w_{m}\cdots w_{1}w_0w_1\cdots w_m$ where each $w_i$ is the longest non-overlapping border of $w_{i}\cdots w_{1}w_0w_1\cdots w_{i}$, except $w_0$, which is either empty or unbordered. For example, going back to Example~\ref{example:abra}, the smallest BP-factorization of ${\tt abracadabra}$ is ${\tt abra}\cdot {\tt cad} \cdot {\tt abra}$ and the smallest BP-factorization of ${\tt reappear}$ is ${\tt r}\cdot {\tt ea} \cdot {\tt p}\cdot {\tt p} \cdot {\tt ea} \cdot {\tt r}$.

A natural question to ask is: what is the maximum possible width $f_k(n)$ of the smallest BP-factorization of a length-$n$ word? Through empirical observation, we arrive at the following conjectures: 
\begin{itemize}
    \item We have $f_2(8n+i)=6n+i$ for $i$ with $0\leq i \leq 5$ and $f_2(8n+6) = f_2(8n+7) = 6n+5$. 
    \item We have $f_k(n) = n$ for $k\geq 3$.
\end{itemize}
To calculate $f_k(n)$, two things are needed: an upper bound on $f_k(n)$, and words that witness the upper bound. 
\begin{theorem}
Let $l\geq 0$ be an integer. Then $f_2(8l+i) = 6l + i$ for $i$ with $0\leq i \leq 5$ and $f_2(8l+6)=f_2(8l+7)= 6l + 5$.
\end{theorem}
\begin{proof}
Let $n\geq 0$ be an integer. We start by proving lower bounds on $f_2(n)$. Suppose $n=8l$ for some $l\geq 0$. Then the width of the smallest BP-factorization of \[(0101)^{l}(1001)^{l}\] is $6l$, so $f_2(8l)\geq 6l$. To see this, notice that the smallest BP-factorization of $01011001$ is $01 \cdot 0 \cdot 1 \cdot 1 \cdot 0 \cdot 01$, and therefore is of width $6$. Suppose $n = 8l+i$ for some $i$ with $1\leq i \leq 7$. Then one can take $(0101)^{l}(1001)^{l}$ and insert either $0$, $00$, $010$, $0110$, $01010$, $010110$, or $0110110$ to the middle of the word to get the desired length.

Now we prove upper bounds on $f_2(n)$. Let $t\leq n$ be a positive integer. Let $w$ be a length-$n$ word whose largest BP-factorization $w_{m}\cdots w_{1}w_0w_1\cdots w_{m}$ is of width $t$. One can readily verify that $f_2(0)=0$, $f_2(1)=1$, $f_2(2)=2$, $f_2(3)=3$, $f_2(4) = 4$, and $f_2(5)=f_2(6)=f_2(7)=5$ through exhaustive search of all binary words of length $<8$. Suppose $m\geq 4$, so $n\geq t \geq 8$. Then we can write $w= w_{m} w_{m-1}w_{m-2}\cdots w_{m-2}w_{m-1}w_m$ where $|w_{m-2}|,|w_{m-1}|,|w_m|>0$. It is easy to show that $|w_{m-2}w_{m-1}w_m| \geq 4$ by checking that all binary words of length $<8$ do not admit a smallest BP-factorization of width $6$. In the worst case, we can peel off prefixes and suffixes of length $4$ while accounting for the $6$ blocks they add to the BP-factorization until we hit the middle core of length $<8$. Thus, we have $f_2(8l+i) \leq 6l + j$ where $j$ is the width of the smallest BP-factorization of the middle core, which is of length $i$. We have already computed $f_2(i)$ for $0\leq i \leq 7$, so the upper bounds follow.
\end{proof}

\begin{theorem}
Let $n\geq 0$ and $k\geq 3$ be integers. Then $f_k(n) = n$.
\end{theorem}
\begin{proof}
Clearly $f_k(n) \leq n$. We prove $f_k(n) \geq n$. If $n$ is divisible by $6$, then consider the word $(012)^{n/6}(210)^{n/6}$. If $n$ is not divisible by $6$, then take  $(012)^{\lfloor n/6\rfloor }(210)^{\lfloor n/6 \rfloor}$ and insert either $0$, $00$, $010$, $0110$, or $01010$ in the middle of the word. When calculating the smallest BP-factorization of the resulting words, it is easy to see that at each step we are removing a border of length $1$. Thus, their largest BP-factorization is of width $n$.
\end{proof}

\section{Equal smallest and largest BP-factorizations}\label{section:uniqueIB}
Recall back to Example~\ref{example:abra}, that ${\tt abracadabra}$ has distinct smallest and largest BP-factorizations, namely ${\tt abra}\cdot {\tt cad} \cdot {\tt abra}$ and ${\tt a} \cdot {\tt br} \cdot {\tt a} \cdot {\tt cad} \cdot {\tt a} \cdot {\tt br} \cdot {\tt a}$. However, the word ${\tt alfalfa}$ has the same smallest and largest BP-factorizations, namely ${\tt a}\cdot {\tt lf} \cdot {\tt  a}\cdot {\tt lf} \cdot {\tt a}$. Under what conditions are the smallest and largest BP-factorizations of a word the same? Looking at unique borders seems like a good place to start, since the shortest border and longest non-overlapping border coincide when a word has a unique border. However, the converse is not true---just consider the previous example ${\tt alfalfa}$. The shortest border and longest non-overlapping border are both ${\tt a}$, but ${\tt a}$ is not a unique border of ${\tt alfalfa}$.

In Theorem~\ref{theorem:SL} we characterize all words whose smallest and largest BP-factorization coincide.

\begin{theorem}\label{theorem:SL}
    Let $m,m'\geq 1$ be integers. Let $w$ be a word with smallest BP-factorization $w_{m'}'\cdots w_{1}'w_0'w_1'\cdots w_{m'}'$ and largest BP-factorization $w_{m}\cdots w_{1}w_0w_1\cdots w_m$. Then $m=m'$ and $w_{i}= w_{i}'$ for all $i$, $0 \leq i \leq m$ if and only if for all $i\neq 2$, $0< i\leq m$, we have that $w_i$ is the unique border of $w_{i}\cdots w_{1}w_0w_1\cdots w_i$ and for $i=2$ we have that either 
    \begin{enumerate}
        \item  $w_2$ is the unique border of $w_2w_1w_0w_1w_2$, or
        \item $w_2w_1w_0w_1w_2 = w_0w_1w_0w_1w_0$ where $w_0$ is the unique border of $w_0w_1w_0$.
    \end{enumerate}
\end{theorem}
\begin{proof}~\\
    $\Longrightarrow:$ Let $i$ be an integer such that $0< i \leq m$. Let $u_i=w_{i}\cdots w_{1}w_0w_1\cdots w_i$. Since $w_i$ is both the shortest border and longest non-overlapping border of $u_i$ (i.e., $w_i=w_i'$), we have that $u_i$ has exactly one border of length $\leq |u_i|/2$. Thus, either $w_{i}$ is the unique border of $u_i$, or $u_i$ has a border of length $>|u_i|/2$. If $w_i$ is the unique border of $u_i$, then we are done. So suppose that $u_i$ has a border of length $>|u_i|/2$. Let $v_i$ be the shortest such border. We have that $w_i$ is both a prefix and suffix of $v_i$. In fact, $w_i$ must be the unique border of $v_i$. Otherwise we contradict the minimality of $v_i$, or the assumption that $w_i$ is both the shortest border and longest non-overlapping border of $u_i$. Since $w_i$ is unbordered, it cannot overlap itself in $v_i$ and $w_i$. So we can write $v_i = w_i y w_i$ for some word $y$ where $u_i = w_iyw_iyw_i$, or $u_i = w_i x w_i x' w_i x'' w_i$ such that $y=xw_ix'=x'w_ix''$. If $u_i =  w_i x w_i x' w_i x'' w_i$, then we see that $w_ix'$ is a suffix of $y$ and $x'w_i$ is a prefix of $y$, implying that $w_ix'w_i$ is a new smaller border of $u_i$. This either contradicts the assumption that $v_i$ is the shortest border of length $>|u_i|/2$, or the assumption that $u_i$ has exactly one border of length $\leq |u_i|/2$. Thus, we have that $u_i = w_iyw_iyw_i$. The shortest border and longest non-overlapping border of $yw_iy$ must be $y$, by assumption. Additionally, $w_i$ is unbordered, so $u_i$ is of width $5$ and $i=2$. This implies that $w_i = w_2 = w_0$ and $y=w_1$. 
    
\bigskip\noindent 
    $\Longleftarrow:$ Let $i$ be an integer such that $0< i\leq m$. We omit the case when $i=0$, since proving $w_i = w_i'$ for all other $i$ is sufficient. Since $w_i$ is the unique border of $u_i = w_i\cdots w_1w_0w_1\cdots w_i$, we have that the shortest border and longest non-overlapping border of $u_i$ is $w_i$. In other words, we have that $w_i = w_i'$. Suppose $i=2$ and $u_2 = w_2w_1w_0w_1w_2 = w_0w_1w_0w_1w_0$ where $w_0$ is the unique border of $w_0w_1w_0$. Since $w_0$ is the unique border of $w_0w_1w_0$, it is also the shortest border of $u_2$. Additionally, the next longest border of $u_2$ is $w_0w_1w_0$, which is overlapping. So $w_0$ is also the longest non-overlapping border of $u_2$. Thus $w_2=w_2'$.
\end{proof}

Just based on this characterization, finding a recurrence for the number of words with a coinciding smallest and largest BP-factorization seems hard. So we turn to a different, related problem: counting the number of words with a unique border.

\subsection{Unique borders}
Harju and Nowotka~\cite{Harju&Nowotka:2005} counted the number $B_k(n)$ of length-$n$ words over $\Sigma_k$ with a unique border, and the number $B_k(n,t)$ of length-$n$ words over $\Sigma_k$ with a length-$t$ unique border. However, through personal communication with the authors, a small error in one of the proofs leading up to their formula for $B_k(n,t)$ was discovered. Thus, the formula for $B_k(n,t)$ as stated in their paper is incorrect. In this section, we present the correct recurrence for the number of length-$n$ words with a length-$t$ unique border. We also show that the probability a length-$n$ word has a unique border tends to a constant. See \href{http://oeis.org/A334600}{\underline{A334600}} in the OEIS~\cite{OEIS} for the sequence $(B_2(n))_{n\geq 0}$.

Suppose $w$ is a word with a unique border $u$. Then $u$ must be unbordered, and $|u|$ must not exceed half the length of $w$. If either of these were not true, then $w$ would have more than one border. By combining these ideas, we get Theorem~\ref{theorem:uniquet} and Theorem~\ref{theorem:unique}.

\begin{theorem}\label{theorem:uniquet}
Let $n>t\geq 1$ be integers. Then the number of length-$n$ words with a unique length-$t$ border satisfies the recurrence 
   \[ B_k(n,t) = \begin{cases} 
      0, & \text{if $n < 2 t$;} \\
      u_tk^{n-2t} - \sum_{i=2t}^{\lfloor n/2\rfloor} B_k(i,t)k^{n-2i}, & \text{if $n \geq 2t$ and $n+t$ odd;}\\
     u_t k^{n-2t} - B_k((n+t)/2,t)-\sum_{i=2t}^{\lfloor n/2\rfloor} B_k(i,t)k^{n-2i}, & \text{if $n \geq 2t$ and $n+t$ even.}
   \end{cases}
\]
\end{theorem}
\begin{proof}
Let $w$ be a length-$n$ word with a unique length-$t$ border $u$. Since $u$ is the unique border of $w$, it is unbordered. Thus, we can write $w=uvu$ for some (possibly empty) word $v$. For $n<2t$, we have that $B_k(n,t)=0$ since $u$ is unbordered and thus cannot overlap itself in $w$.

Suppose $n\geq 2t$. Let $\overline{B_k}(n,t)$ denote the number of length-$n$ words that have a length-$t$ unbordered border and have a border of length $>t$. Clearly $B_k(n,t) = u_tk^{n-2t} - \overline{B_k}(n,t)$. Suppose $w$ has another border $u'$ of length $>t$. Furthermore, suppose that there is no other border $u''$ with $|u| < |u''| < |u'|$. Then $u$ is the unique border of $u'$. Since $u$ is the shortest border, we have $|u| \leq n/2$. But we could possibly have $|u'|> n/2$. The only possible way for $|u'|$ to exceed $n/2$ is if $w=uv'uv'u$ for some (possibly empty) word $v$. But this is only possible if $n+t$ is even; otherwise we cannot place $u$ in the centre of $w$. When $n+t$ is odd, we compute $\overline{B_k}(n,t)$ by summing over all possibilities for $u'$ (i.e., $2t\leq |u'| \leq \lfloor n/2\rfloor$) and the middle part of $w$ (i.e., $v''$ where $w=u'v''u'$). This gives us the recurrence, \[\overline{B_k}(n,t) =\sum_{i=2t}^{\lfloor n/2\rfloor} B_k(i,t)k^{n-2i}.\]
When $n+t$ is even, we compute $\overline{B_k}(n,t)$ in the same fashion, except we also include the case where $|u'| = (n+t)/2$. This gives us the recurrence,
\[\overline{B_k}(n,t) =B_k((n+t)/2,t)+\sum_{i=2t}^{\lfloor n/2\rfloor} B_k(i,t)k^{n-2i}.\]
\end{proof}
\begin{theorem}\label{theorem:unique}
Let $n\geq  2$ be an integer. Then the number of length-$n$ words with a unique border is 
   \[ B_k(n) = \sum_{t=1}^{\lfloor n/2\rfloor}B_k(n,t).\]
\end{theorem}
\subsection{Limiting values}
We show that the probability that a random word of length $n$ has a unique border tends to a constant. Table~\ref{table:BProbTable} shows the behaviour of this probability as $k$ increases.

Let $P_{n,k}$ be the probability that a random word of length $n$ has a unique border.  Then

\[P_{n,k}=\frac{B_k(n)}{k^n}=\frac{1}{k^n} \sum_{i=1}^{\lfloor n/2\rfloor} B_k(n,i).\]
\begin{lemma}\label{lemma:UBconv}
Let $k\geq 2$ and $n\geq 2t\geq 2$ be integers. Then \[ \frac{B_k(n,t)}{k^n}\leq\frac{1}{k^{t}}.\]
\end{lemma}
\begin{proof}
Let $w$ be a length-$n$ word. Suppose $w$ has a unique border of length $t$. Since $t\leq n/2$, we can write $w=uvu$ for some words $u$ and $v$ where $|u|=t$. But this means that $B_k(n,t)\leq k^{n-t}$, and the lemma follows.
\end{proof}

\begin{theorem}
Let $k\geq 2$ be an integer. Then the limit $P_k = \lim\limits_{n\to \infty} P_{n,k}$ exists.
\end{theorem}
\begin{proof}
Follows from the definition of $P_{n,k}$, Lemma~\ref{lemma:UBconv}, and the direct comparison test for convergence.
\end{proof}

\begin{table}[H]

\centering
\begin{tabular}{|c|c|}
\hline
$k$ & $\approx  P_{k}$  \\
\hline
2 & 0.5155 \\ 
3 & 0.3910  \\ 
4 & 0.2922   \\ 
5 & 0.2302  \\ 
6 & 0.1890  \\ 
7 & 0.1599  \\ 
8 & 0.1384 \\ 
9 & 0.1219  \\ 
10 & 0.1089 \\ 
\vdots & \vdots  \\
100 & 0.0101 \\
\hline
\end{tabular}
\captionsetup{justification=centering}
\caption{Probability that a word has a unique border.}
\label{table:BProbTable}
\end{table}

\bibliographystyle{unsrt}
\bibliography{abbrevs,main}

\end{document}